\documentclass[11pt]{amsart}
\usepackage[T1]{fontenc}
\usepackage{lmodern}
\usepackage[margin=1in]{geometry}
\usepackage{amsmath,amssymb,amsthm,mathtools,microtype}
\usepackage[hidelinks]{hyperref}
\hypersetup{pdftitle={Maximal growth of the Stein-Wainger oscillatory integral},
pdfauthor={Cheng Zhang and Zhifei Zhu},
pdfsubject={Optimal growth in Denjoy-Carleman classes},
pdfkeywords={oscillatory integral, Denjoy-Carleman, Gevrey, quasianalyticity}}
\newtheorem{theorem}{Theorem}[section]
\newtheorem{proposition}[theorem]{Proposition}
\newtheorem{corollary}[theorem]{Corollary}
\newtheorem{lemma}[theorem]{Lemma}
\theoremstyle{remark}

\newcommand{\CM}{\mathcal C_M}
\newcommand{\R}{\mathbb R}

\newcommand{\E}{\mathbb E}
\newcommand{\Prob}{\mathbb P}

\numberwithin{equation}{section}
\title{Maximal growth of the Stein--Wainger oscillatory integral}
\author{ Cheng Zhang and Zhifei Zhu}

\address{Yau Mathematical Sciences Center\\
	Tsinghua University\\
	Beijing, BJ 100084, China}
\email{czhang98@tsinghua.edu.cn; zhifeizhu@tsinghua.edu.cn}
\date{}
\subjclass[2020]{Primary 42B20. Secondary 26E10, 42B15.}
\keywords{Singular oscillatory integral, Denjoy--Carleman class,
quasianalyticity, Gevrey class, eigenfunction restriction}

\begin{document}
\begin{abstract}
We determine the optimal growth of the Stein--Wainger oscillatory
integral for every Denjoy--Carleman class defined by a positive
log-convex sequence. Quasianalyticity is equivalent to boundedness
in $\lambda$ for every fixed phase in the class. In the
nonquasianalytic case, an explicit formula in the defining sequence
gives the optimal order, attained by a single
phase for all sufficiently large frequencies. This yields sharp
bounds for Gevrey classes, their logarithmic refinements, and classes
beyond every Gevrey class. For general smooth phases the growth is
$o(\log\lambda)$, and every prescribed $h(\lambda)\to\infty$ with
$h(\lambda)=o(\log\lambda)$ is attained asymptotically along a sequence
by one smooth phase.
These results answer a question of Wang and Zhang motivated by
endpoint eigenfunction restriction.
\end{abstract}
\maketitle

\section{Introduction}\label{sec:intro}

Let $I=[-1,1]$. For a real-valued phase $\psi\in C^1(I)$, define
\begin{equation}\label{eq:m}
m_\psi(\lambda)
=\operatorname{p.v.}\int_{-1}^1 e^{i\lambda\psi(t)}\,\frac{dt}{t},
\qquad \lambda\ge2.
\end{equation}
The principal value exists since
\begin{equation}\label{eq:pv}
m_\psi(\lambda)
=\int_0^1\frac{e^{i\lambda\psi(t)}-e^{i\lambda\psi(-t)}}{t}\,dt.
\end{equation}
We study the largest possible growth of $|m_\psi(\lambda)|$ for a fixed
phase in a prescribed regularity class. Throughout, implicit constants
may depend on the phase and the fixed class parameters, but not on
$\lambda$. We write $a\approx b$ for two-sided bounds by positive
constants, and $a\sim b$ when $a/b\to1$.

Stein and Wainger \cite{sw70,sw78} proved that for a real polynomial $P$
of degree at most $d$,
\[
\left|\operatorname{p.v.}\int_{-1}^1e^{iP(t)}\,\frac{dt}{t}\right|
\le C_d,
\]
with $C_d$ independent of the coefficients. Parissis
\cite[(3.7), (4.4)]{pa08} showed that the optimal dependence is
$C_d\approx\log(d+1)$. Analogous bounds for rational phases, depending
only on the degrees of numerator and denominator, were obtained by
Folch-Gabayet and Wright \cite{fw03}. For general smooth phases,
Wang and Zhang \cite[Problem 4, Section 4]{wz21} asked whether
$m_\psi(\lambda)=O(1)$. We give a negative answer and determine the
optimal growth throughout the Denjoy--Carleman scale.

The question arises from singular oscillatory operators of the form
\begin{equation}\label{eq:operator}
T_\lambda f(t)
=\operatorname{p.v.}\int_\R
e^{i\lambda\phi(t,s)}\frac{a(t,s)f(s)}{t-s}\,ds,
\qquad a\in C_c^\infty(\R^2),
\end{equation}
where $\phi$ is real-valued and smooth. These operators occur in Hilbert
transforms along curves and endpoint restriction estimates.
Relevant $L^2$ estimates were developed by Phong--Stein
\cite{ps86}, Ricci--Stein \cite{rs87}, Pan \cite{pan91,pan93}, Seeger
\cite{seeger94}, and Carbery--P\'erez \cite{cp99}. The
translation-invariant model has phase $\phi(t,s)=\psi(t-s)$ and
truncated convolution kernel
$e^{i\lambda\psi(t)}\mathbf1_{[-1,1]}(t)/t$. Its Fourier multiplier
has value $m_\psi(\lambda)$ at zero.

More specifically, let $(X,g)$ be a smooth closed Riemannian manifold
of dimension at least three, let $\Sigma\subset X$ be a closed embedded
submanifold of codimension two, and let
$-\Delta_g e_\lambda=\lambda^2e_\lambda$. Burq--G\'erard--Tzvetkov
\cite{BGT} established the endpoint bound
\begin{equation}\label{eq:restriction}
\|e_\lambda\|_{L^2(\Sigma)}
\le C_{X,g,\Sigma}\lambda^{1/2}\sqrt{\log\lambda}
\,\|e_\lambda\|_{L^2(X)}.
\end{equation}
Chen--Sogge \cite{chensogge} removed the logarithm for geodesics in
dimension three. Wang--Zhang \cite{wz21} extended this to totally
geodesic codimension-two submanifolds in all dimensions and to
finite-type curves in dimension three. Their $TT^*$ argument
leads to operators of the form \eqref{eq:operator}.
In their recent work \cite[Theorems 1 and 2]{wz26}, Wang--Zhang proved
that for every fixed smooth $\Sigma$,
\begin{equation}\label{eq:restriction-new}
\|e_\lambda\|_{L^2(\Sigma)}
=o\!\left(\lambda^{1/2}\sqrt{\log\lambda}\right)
\|e_\lambda\|_{L^2(X)},
\end{equation}
and that no prescribed rate of improvement is valid in general,
already for suitable smooth curves on $S^3$. Finite differentiability
permits the full logarithmic loss,
whereas the bound is $O(\lambda^{1/2})$ when both the Riemannian
structure and the embedding are real analytic. Our results describe
the analogous distinction for \eqref{eq:m} and determine the intermediate
growth rates for every prescribed Denjoy--Carleman class. The common
analytic feature is the singularity $(t-s)^{-1}$ after $TT^*$, with
variable phases and operators acting in the remaining variables
additionally present in the restriction problem.

There is also a direct connection with the Hilbert transform along
the planar curve $(t,\psi(t))$,
\[
\mathcal H_\psi f(x_1,x_2)
=\operatorname{p.v.}\int_{-1}^1
f(x_1-t,x_2-\psi(t))\,\frac{dt}{t}.
\]
Its multiplier is
$\operatorname{p.v.}\int_{-1}^1e^{-i(\xi t+\eta\psi(t))}\,dt/t$,
whose value at $(0,-\lambda)$ is $m_\psi(\lambda)$. The multiplier is
continuous, so our unbounded examples give curves for which
$\mathcal H_\psi$ is unbounded on $L^2(\R^2)$. These examples complement
the classical work of Nagel--Wainger \cite{nw} and
Nagel--Vance--Wainger--Weinberg \cite{nvww}.

\subsection{Main results}

Finite differentiability allows logarithmic growth.
\begin{proposition}\label{prop:finite}
Let $\alpha\ge1$ be an integer. For every real-valued
$\psi\in C^\alpha(I)$,
\[
m_\psi(\lambda)=O(\log\lambda).
\]
There exists a real-valued $\psi\in C^\alpha(I)$ such that
$|m_\psi(\lambda)|\approx\log\lambda$ for all sufficiently large
$\lambda$.
\end{proposition}

Smoothness improves this bound, and every divergent
$h=o(\log\lambda)$ can be attained along a sequence.
\begin{theorem}\label{thm:smooth}
For every real-valued $\psi\in C^\infty(I)$,
\[
m_\psi(\lambda)=o(\log\lambda).
\]
For every function $h:[2,\infty)\to(0,\infty)$ satisfying
$h(\lambda)\to\infty$ and $h(\lambda)=o(\log\lambda)$, there exist
a real-valued $\psi\in C^\infty(I)$ and a sequence
$\lambda_n\to\infty$ such that
\begin{equation}\label{eq:smooth-sharp}
\frac{|m_\psi(\lambda_n)|}{h(\lambda_n)}\longrightarrow1.
\end{equation}
The phase may be chosen flat at zero and identically zero on $[-1,0]$.
\end{theorem}
In particular, no
prescribed $o(\log\lambda)$ upper bound holds for all smooth phases,
even with constants depending on the phase.

We now give the optimal bound for a fixed Denjoy--Carleman class.
Let $M=(M_n)_{n\ge0}$ be a positive log-convex sequence with $M_0=1$,
so $M_n^2\le M_{n-1}M_{n+1}$ for $n\ge1$. A function
$f\in C^\infty(I)$ belongs to $\CM(I)$ if, for some $K\ge1$,
\begin{equation}\label{eq:class}
\|f^{(n)}\|_{L^\infty(I)}\le K^{n+1}M_n\qquad(n\ge0).
\end{equation}
Set
\begin{equation}\label{eq:tail}
\mu_n=\frac{M_n}{M_{n-1}},\qquad
T_M(N)=\sum_{j\ge N}\frac1{\mu_j}\quad(N\ge1).
\end{equation}
The quotients $\mu_n$ are nondecreasing. By the Denjoy--Carleman
theorem \cite[Theorem 1.3.8]{hor1}, the class is quasianalytic exactly
when $T_M(1)=\infty$, meaning that a function in the class whose
derivatives all vanish at one point is identically zero.

\begin{theorem}\label{thm:dc}
Let $M$ be as above and let $\psi\in\CM(I)$ be real-valued.
\begin{enumerate}
\item If $T_M(1)=\infty$, then $m_\psi(\lambda)=O(1)$.
\item If $T_M(1)<\infty$, define, for sufficiently large $L$,
\begin{equation}\label{eq:rate}
R_M(L)=\min\left\{
\inf_{n\ge1}\frac{L+\log(M_n/n!)}{n},\
\log\frac1{T_M(\lfloor L\rfloor+1)}\right\}.
\end{equation}
Then $R_M(L)\to\infty$, $R_M(L)=o(L)$, and
\begin{equation}\label{eq:dc-upper}
|m_\psi(\lambda)|\le C_\psi R_M(\log\lambda)
\end{equation}
for all sufficiently large $\lambda$. Moreover, there exists a
real-valued $\psi\in\CM(I)$, zero on $[-1,0]$ and strictly increasing
and convex on $[0,1]$, such that
\begin{equation}\label{eq:dc-sharp}
|m_\psi(\lambda)|\approx R_M(\log\lambda)
\end{equation}
for all sufficiently large $\lambda$.
\end{enumerate}
\end{theorem}

Theorem~\ref{thm:dc} shows that quasianalyticity is equivalent to
boundedness of $m_\psi$ for every fixed phase in $\CM$.
In each nonquasianalytic class, \eqref{eq:rate} determines the optimal
growth directly from $M$. The two terms arise from Taylor's
formula and Bang's estimate \cite[Theorem~I, p.~141]{bang}, in the
form stated in Lemma~\ref{lem:bang}, and both are essential.
For sharpness, we combine infinite
convolutions of uniform densities, as in H\"ormander
\cite[Theorem~1.3.5]{hor1}, with the lower bound
\eqref{eq:sharp-scale} to obtain a single convex phase attaining
this order for all sufficiently large frequencies.

\subsection{Consequences for specific classes}

For $s\ge1$, the Gevrey class $G^s(I)$ is defined by
$M_n=(n!)^s$. In particular, $G^1(I)$ is the real-analytic class.
We refer to Rodino \cite{rod93} for its role in partial differential
equations.
\begin{corollary}\label{cor:gevrey}
If $\psi\in G^1(I)$ is real-valued, then $m_\psi(\lambda)=O(1)$.
For every $s>1$ and real-valued $\psi\in G^s(I)$,
\[
m_\psi(\lambda)=O(\log\log\lambda).
\]
For each fixed $s>1$, there is one real-valued phase in $G^s(I)$ with
$|m_\psi(\lambda)|\approx\log\log\lambda$ for all sufficiently large
$\lambda$.
\end{corollary}

Write $\log^{(k)}$ and $\exp^{(k)}$ for the $k$-fold iterates.
For $k\ge1$ and $s\ge1$, set, for large $n$,
\begin{equation}\label{eq:Q}
Q_{k,s}(n)=(\log^{(k)}n)^s\prod_{j=1}^{k-1}\log^{(j)}n,
\qquad M_n=n!Q_{k,s}(n)^n.
\end{equation}
The initial terms are chosen so that $M_0=1$ and $M$ is positive
and log-convex.
Let $G_k^s(I)$ denote the corresponding class. These classes lie
between the analytic class and every Gevrey class of order greater
than one. They are quasianalytic exactly when $s=1$.
\begin{corollary}\label{cor:refined}
Let $k\ge1$. If $\psi\in G_k^1(I)$ is real-valued, then
$m_\psi(\lambda)=O(1)$. For every $s>1$ and real-valued
$\psi\in G_k^s(I)$,
\[
m_\psi(\lambda)=O(\log^{(k+2)}\lambda).
\]
For each fixed $k\ge1$ and $s>1$, there is one real-valued phase in
$G_k^s(I)$ with $|m_\psi(\lambda)|\approx\log^{(k+2)}\lambda$
for all sufficiently large $\lambda$.
\end{corollary}

Larger Denjoy--Carleman classes also arise in applications. For
instance, J\'ez\'equel \cite{jez} used ultradifferentiable classes
beyond Gevrey regularity in the study of trace formulae for Anosov
flows. The next two corollaries give sharp growth rates in this range.
\begin{corollary}\label{cor:fast}
Fix $c>0$ and $\alpha>1$, and let $M_n=e^{cn^\alpha}$.
Every real-valued $\psi\in\CM(I)$ satisfies
\[
m_\psi(\lambda)=O\!\left((\log\lambda)^{1-1/\alpha}\right).
\]
For these fixed $c$ and $\alpha$, one real-valued phase in $\CM(I)$
satisfies $|m_\psi(\lambda)|\approx(\log\lambda)^{1-1/\alpha}$
for all sufficiently large $\lambda$.
\end{corollary}

\begin{corollary}\label{cor:iterated}
Fix $k\ge2$ and $c,\alpha>0$, and let
$M_n=\exp^{(k)}(cn^\alpha)$ for all sufficiently large $n$, with
initial terms chosen as above. Every real-valued $\psi\in\CM(I)$
satisfies
\[
m_\psi(\lambda)
=O\!\left(\frac{\log\lambda}
{(\log^{(k)}\lambda)^{1/\alpha}}\right).
\]
For these fixed $k,c,\alpha$, one real-valued phase in $\CM(I)$
satisfies $|m_\psi(\lambda)|\approx
\log\lambda/(\log^{(k)}\lambda)^{1/\alpha}$
for all sufficiently large $\lambda$.
\end{corollary}

\subsection*{Organization of the paper}
In Section~\ref{sec:smooth}, we prove Proposition~\ref{prop:finite}
and Theorem~\ref{thm:smooth}. We establish the upper bounds for
finitely differentiable and smooth phases and give the corresponding
sharpness constructions, including smooth phases with prescribed
sublogarithmic growth along a sequence. In Section~\ref{sec:dc}, we prove Theorem~\ref{thm:dc}. We first
recall Bang's estimate and treat the quasianalytic case. We then
derive the optimal rate in the nonquasianalytic case and construct
a single phase attaining it for all sufficiently large frequencies.
We conclude with the proofs of
Corollaries~\ref{cor:gevrey}--\ref{cor:iterated}.

\subsection*{Acknowledgments}
The authors thank Shaozhen Xu for helpful discussions. The authors
are partially supported by the National Key R\&D Program of China,
Grant No.~2024YFA1015300. C.Z. is also partially supported by NSFC
Grant No.~12371097, and Z.Z. by NSFC Grant No.~12501065.

\section{Finite differentiability and smooth phases}\label{sec:smooth}

In the smooth case, after subtracting $\psi(0)$, the finite-type
estimate reduces the upper bound to phases flat at zero, for which Taylor's formula gives
$o(\log\lambda)$. For sharpness, we arrange disjointly supported
dilates whose contributions agree at selected frequencies while the
resulting phase remains smooth and flat at zero. For each prescribed
divergent $h=o(\log\lambda)$, this gives one phase with
$|m_\psi(\lambda_n)|\sim h(\lambda_n)$ along a sequence
$\lambda_n\to\infty$, so no universal rate of improvement follows
from smoothness alone.

\subsection{Finite differentiability and the finite-type reduction}

\begin{proof}[Proof of Proposition~\ref{prop:finite}]
Since $|\psi(t)-\psi(-t)|\le Ct$, \eqref{eq:pv} gives
\[
|m_\psi(\lambda)|
\le C\int_0^1\min\{1,\lambda t\}\,\frac{dt}{t}
\le C(1+\log\lambda).
\]
For sharpness, take $\psi(t)=(\max\{t,0\})^{\alpha+1}$.
This function belongs to $C^\alpha(I)$, and
\[
m_\psi(\lambda)
=\frac1{\alpha+1}\int_0^\lambda\frac{e^{iu}-1}{u}\,du
=-\frac1{\alpha+1}\log\lambda+O(1).
\]
The last equality follows by splitting at $u=1$ and integrating
$e^{iu}/u$ by parts on $[1,\lambda]$.
\end{proof}

The next lemma is a finite-type estimate. We give a direct proof;
the classical weighted van der Corput estimate can be found in
\cite[Chapter~VIII, \S1.2, Corollary, p.~334]{stein93}.

\begin{lemma}\label{lem:finite}
Let $\psi\in C^\infty(I)$ be real-valued. If $\psi-\psi(0)$ is not
flat at zero, then $m_\psi(\lambda)=O(1)$.
\end{lemma}
\begin{proof}
Let $k\ge1$ be the least order with $\psi^{(k)}(0)\ne0$.
Taylor's formula gives $\delta\in(0,1)$ and $c,C>0$ such that,
for $0<t\le\delta$,
\[
|\psi(t)-\psi(-t)|\le Ct^k,\qquad
|\psi'(\pm t)|\ge ct^{k-1},\qquad
|\psi'(\pm t)|+t|\psi''(\pm t)|\le Ct^{k-1}.
\]
For large $\lambda$, set $r=\lambda^{-1/k}\le\delta$. Then
\[
\left|\int_0^r
\frac{e^{i\lambda\psi(t)}-e^{i\lambda\psi(-t)}}t\,dt\right|
\le C\lambda\int_0^r t^{k-1}\,dt\le C.
\]
On $[r,\delta]$, the derivative estimates give
\[
\left|\frac1{t\psi'(\pm t)}\right|\le Ct^{-k},\qquad
\left|\frac{d}{dt}\frac1{t\psi'(\pm t)}\right|\le Ct^{-k-1}.
\]
One integration by parts therefore yields, for either sign,
\[
\left|\int_r^\delta e^{i\lambda\psi(\pm t)}\,\frac{dt}{t}\right|
\le\frac C\lambda
\left(r^{-k}+\int_r^\delta t^{-k-1}\,dt\right)\le C.
\]
The remaining integral in \eqref{eq:pv} is bounded by
$2\log(1/\delta)$.
\end{proof}

\subsection{General smooth phases}

\begin{proof}[Proof of Theorem~\ref{thm:smooth}]
Subtracting $\psi(0)$ preserves $|m_\psi|$. By
Lemma~\ref{lem:finite}, we may therefore assume that $\psi$ is flat
at zero. For every integer $N\ge1$, Taylor's formula gives
$C_N\ge1$ such that $|\psi(t)-\psi(-t)|\le C_Nt^N$.
With $r=(\lambda C_N)^{-1/N}$, we obtain
\[
|m_\psi(\lambda)|
\le\lambda C_N\int_0^r t^{N-1}\,dt+2\int_r^1\frac{dt}{t}
=\frac1N+\frac2N\log(\lambda C_N).
\]
Divide by $\log\lambda$, let $\lambda\to\infty$, and then let
$N\to\infty$.

For sharpness, choose nonzero $\eta\in C_c^\infty((1/2,1))$ with
$0\le\eta\le1$, and put
\[
A=\int_{1/2}^1\frac{e^{i\eta(t)}-1}{t}\,dt.
\]
Then $\Re A<0$. Set $\lambda_0=1$ and $N_0=0$. Recursively choose
$\lambda_n>2\lambda_{n-1}$ and set
$N_n=\lceil h(\lambda_n)/|A|\rceil$ so that
\begin{equation}\label{eq:smooth-choice}
N_n\ge nN_{n-1},\qquad
nN_n\le\log\lambda_n,\qquad
\frac{\lambda_{n-1}N_n}{\lambda_n}\le2^{-n}.
\end{equation}
All three conditions are possible because $h(\lambda)\to\infty$
and $h(\lambda)=o(\log\lambda)$.
Define, on $\R$,
\begin{equation}\label{eq:smooth-phase}
\psi(t)=\sum_{n\ge1}\frac1{\lambda_n}
\sum_{j=N_{n-1}+1}^{N_n}\eta(2^jt).
\end{equation}
The supports are disjoint, and the sum is locally finite away from
zero. If $N_{n-1}<j\le N_n$ and $2^{-j-1}<t<2^{-j}$, then for every
fixed integer $k\ge0$,
\[
|\psi^{(k)}(t)|
\le\|\eta^{(k)}\|_\infty\frac{2^{kj}}{\lambda_n}
\le\|\eta^{(k)}\|_\infty
e^{-(n-k\log2)N_n}\longrightarrow0.
\]
Thus $\psi\in C_c^\infty(\R)$, is flat at zero, and vanishes for
$t\le0$.

At frequency $\lambda_n$, each interval indexed by
$N_{n-1}<j\le N_n$ contributes exactly $A$ after the substitution
$u=2^jt$. On the other intervals use
$|e^{ix}-1|\le\min\{2,|x|\}$ to obtain
\begin{align*}
m_\psi(\lambda_n)
&=A(N_n-N_{n-1})+O(N_{n-1})
  +O\!\left(\lambda_n\sum_{\ell>n}\frac{N_\ell}{\lambda_\ell}\right)\\
&=AN_n+O(N_{n-1}+2^{-n}),
\end{align*}
where
\[
\lambda_n\sum_{\ell>n}\frac{N_\ell}{\lambda_\ell}
\le\sum_{\ell>n}
\frac{\lambda_{\ell-1}N_\ell}{\lambda_\ell}
\le\sum_{\ell>n}2^{-\ell}=2^{-n}.
\]
Since $N_{n-1}/N_n\le1/n$ and
$|A|N_n/h(\lambda_n)\to1$, this proves \eqref{eq:smooth-sharp}.
\end{proof}

\section{Proof of Theorem~\ref{thm:dc}}\label{sec:dc}

In the nonquasianalytic case, Taylor's formula and Bang's estimate
give the two terms in \eqref{eq:rate}, and \eqref{eq:Rscale} completes
the upper bound. For sharpness, we use infinite convolutions of
uniform densities, as in H\"ormander \cite[Theorem~1.3.5]{hor1}, to
construct a convex phase with the required derivative bounds.
The lower bound \eqref{eq:sharp-scale} and integration by parts then
give the optimal growth for all sufficiently large frequencies.

\subsection{Bang's estimate and the quasianalytic case}

Subtracting $\psi(0)$ preserves $|m_\psi|$ and membership in $\CM$.
By Lemma~\ref{lem:finite}, it remains to consider phases flat at zero.
We use the following local form of Bang's estimate
\cite[Theorem~I, p.~141, and (14), p.~142]{bang}.
The constants and proof follow the log-convex case of
\cite[Lemma~1, (2.1$'$), pp.~1082--1083]{boman}.

\begin{lemma}\label{lem:bang}
Let $f\in C^\infty([0,r])$ be flat at zero and satisfy
$\|f^{(n)}\|_\infty\le A_n$, where $A_n>0$ and
$\eta_n=A_n/A_{n-1}$ is nondecreasing. If $\ell\ge1$, $0<x\le r$, and
$\sum_{j\ge\ell}\eta_j^{-1}>4x$, then $|f(x)|\le A_0\,2^{-\ell}$.
\end{lemma}
\begin{proof}
Choose $q\ge\ell$ with $\sum_{j=\ell}^q\eta_j^{-1}\ge4x$ and set
$a=x/(\sum_{j=\ell}^q\eta_j^{-1})\le1/4$. Partition $[0,x]$ by
$x=t_\ell>\cdots>t_{q+1}=0$, with $t_j-t_{j+1}=a/\eta_j$.
Write $F(p,j)=\sup_{[t_{j+1},t_j]}|f^{(p)}|$ and $F(p,q+1)=0$.
The fundamental theorem of calculus gives
\[
F(p,j)\le F(p,j+1)+\frac a{\eta_j}F(p+1,j).
\]
Starting from $F(j,j)\le A_j$, descending induction in $j$ and then
in $p$ yields
\[
F(p,j)\le(2a)^{j-p}A_p\qquad(\ell\le j\le q,\ 0\le p\le j).
\]
Indeed, the induction step for $p<j$ uses
$2a+\eta_{p+1}/(2\eta_j)\le2a+1/2\le1$.
Taking $(p,j)=(0,\ell)$ proves the lemma.
\end{proof}

If $T_M(1)=\infty$ and $\psi$ is flat at zero, apply
Lemma~\ref{lem:bang} to $\psi(\pm t)$ with $A_n=K^{n+1}M_n$,
so that $\eta_n=K\mu_n$. Letting $\ell\to\infty$ gives $\psi=0$
on $I$, and hence $m_\psi=0$. This proves the quasianalytic assertion.

\subsection{Direct estimates for the rate}

Assume henceforth that $T_M(1)<\infty$. Replacing $M_n$ by
$T_M(1)^nM_n$ leaves $\CM$ unchanged and adds $\log T_M(1)$ to
both entries in \eqref{eq:rate}. We use the same notation for the
normalized sequence and assume throughout the rest of this section that
$T_M(1)=1$. Write
\[
A(L)=\inf_{n\ge1}\frac{L+\log(M_n/n!)}n,\qquad
B(L)=\log\frac1{T_M(\lfloor L\rfloor+1)},
\]
so that $R_M(L)=\min\{A(L),B(L)\}$.

\begin{lemma}\label{lem:rate}
The function $R_M$ is nondecreasing, tends to infinity, and satisfies
$R_M(L)=o(L)$. Moreover,
\begin{equation}\label{eq:Rscale}
R_M(2L)\le2R_M(L)+4\qquad(L\ge1).
\end{equation}
\end{lemma}
\begin{proof}
Normalization and monotonicity give $\mu_n\ge n$ and $M_n\ge n!$.
Also
\[
\frac{n}{2\mu_n}
\le\sum_{\lfloor n/2\rfloor<j\le n}\frac1{\mu_j}\longrightarrow0.
\]
Thus $\mu_n/n\to\infty$ and $\log(M_n/n!)/n\to\infty$.
It follows that $A(L)\to\infty$: the expressions defining $A(L)$
are large for all sufficiently large $n$, uniformly in $L\ge1$,
while each of the finitely many remaining expressions tends to
infinity with $L$. Since $B(L)\to\infty$ as well,
$R_M(L)\to\infty$. Both entries are nondecreasing. For each fixed
$n$, the bound $R_M(L)\le(L+\log(M_n/n!))/n$ gives
$\limsup_{L\to\infty}R_M(L)/L\le1/n$, proving the little-$o$ assertion.

To prove \eqref{eq:Rscale}, first note that $A(2L)\le2A(L)$.
Set $q=\lfloor L\rfloor+1$ and $p=\lfloor2L\rfloor+1$, so that
$p-q\le q$. If $T_M(p)\ge T_M(q)/2$, then
$B(2L)\le B(L)+\log2$. Otherwise,
\[
\frac{T_M(q)}2
<\sum_{j=q}^{p-1}\frac1{\mu_j}\le\frac q{\mu_q},
\]
and hence $\mu_q\le2q/T_M(q)$. Using $M_q\le\mu_q^q$ and
$q!\ge(q/e)^q$, we obtain
\[
A(2L)\le\frac{2L}{q}+\log\frac{\mu_q}{q}+1
\le B(L)+3+\log2<B(L)+4.
\]
Use $A(2L)\le2A(L)$ when $R_M(L)=A(L)$ and the preceding
two cases when $R_M(L)=B(L)$ to obtain \eqref{eq:Rscale}.
\end{proof}

Since $R_M(L)\to\infty$, the additive change under normalization
does not affect the conclusions or estimates up to constant factors.

\subsection{The upper bound}

Suppose that $\psi$ is flat at zero, choose $K\ge1$ in
\eqref{eq:class}, and put $L=\log\lambda$. Taylor's formula gives
\[
|\psi(t)-\psi(-t)|\le\frac{2K^{n+1}M_n}{n!}t^n
\qquad(n\ge1,\ 0\le t\le1).
\]
Split \eqref{eq:pv} at
$r=(2\lambda K^{n+1}M_n/n!)^{-1/n}\le1$. Then
\[
|m_\psi(\lambda)|
\le\frac1n+\frac2n\log\frac{2\lambda K^{n+1}M_n}{n!}
\le2\,\frac{L+\log(M_n/n!)}n+C_\psi.
\]
Taking the infimum over $n$ gives $|m_\psi(\lambda)|\le2A(L)+C_\psi$.

For the second bound, let $\ell=\lfloor2L\rfloor$ and
$r=T_M(\ell+1)/(8K)$. Apply Lemma~\ref{lem:bang} to $\psi'(\pm t)$
with derivative bounds $K^{n+2}M_{n+1}$. The consecutive quotients
are $K\mu_{n+1}$, and their reciprocals satisfy
\[
\sum_{j\ge\ell}\frac1{K\mu_{j+1}}=\frac{T_M(\ell+1)}K=8r.
\]
Thus
\[
\sup_{|t|\le r}|\psi'(t)|\le K^2M_1\,2^{-\ell}.
\]
Since $\lambda2^{-\ell}\le2\lambda^{1-2\log2}\le2$, we obtain
\[
\lambda\int_0^r\frac{|\psi(t)-\psi(-t)|}{t}\,dt
\le2K^2M_1\lambda r\,2^{-\ell}\le C_\psi.
\]
The remaining integral is bounded by
$2\log(1/r)=2B(2L)+2\log(8K)$. Combining the two estimates,
using $A(L)\le A(2L)$ and \eqref{eq:Rscale}, gives
\[
|m_\psi(\lambda)|
\le2R_M(2L)+C_\psi
\le4R_M(L)+C_\psi
\le C_\psi R_M(L)
\]
for all sufficiently large $\lambda$. Lemma~\ref{lem:finite}
and $R_M(L)\to\infty$ give the same conclusion for nonflat phases.

\subsection{Sharpness}

The construction below uses infinite convolutions of uniform
densities, as in H\"ormander \cite[Theorem~1.3.5]{hor1}.
Let $X_j$ be independent random variables uniformly distributed
on $[0,\mu_j^{-1}]$. Their sum $S=\sum_{j\ge1}X_j$ converges and satisfies
$S\le T_M(1)=1$ almost surely. Define, on $\R$,
\begin{equation}\label{eq:sharp-phase}
\psi(t)=\mu_1\,\E\left(t-\sum_{j\ge2}X_j\right)_+,
\qquad x_+=\max\{x,0\}.
\end{equation}
This function is convex and vanishes for $t\le0$.
Let $k_j(t)=\mu_j\mathbf1_{[0,\mu_j^{-1}]}(t)$ be the density of $X_j$.
Write $D$ for distributional differentiation and
$\|\cdot\|_{\mathrm{TV}}$ for the total variation norm of a finite
signed measure. Then
\[
Dk_j=\mu_j(\delta_0-\delta_{\mu_j^{-1}}),\qquad
\|Dk_j\|_{\mathrm{TV}}=2\mu_j,\qquad
\|k_j\|_\infty=\mu_j.
\]
For $n\ge2$, obtain $\psi^{(n)}$ by differentiating the factors
$k_2,\ldots,k_{n-1}$ in the density of $\sum_{j\ge2}X_j$,
and bound the remaining convolution by $\|k_n\|_\infty$. This gives
\begin{equation}\label{eq:sharp-derivatives}
\|\psi^{(n)}\|_\infty\le2^{n-2}M_n\quad(n\ge2),\qquad
\|\psi'\|_\infty\le\mu_1,\qquad
\|\psi\|_{L^\infty(I)}\le\mu_1.
\end{equation}
Each derivative with $n\ge2$ is represented by a continuous
function: $k_n*k_{n+1}$ is bounded and uniformly continuous, and
convolution with finite signed measures preserves these properties.
Consequently $\psi$ is smooth on $\R$, flat at zero, and belongs
to $\CM(I)$.

Conditioning on the variables with indices $j\ge2$ gives
\[
\Prob(S\le t)
=\E\min\left\{1,\mu_1\left(t-\sum_{j\ge2}X_j\right)_+\right\}
\le\psi(t).
\]
For $L\ge1$, set $n=\lfloor L\rfloor$ and $r=e^{-R_M(L)}$.
The definition of $R_M$ gives directly
\[
T_M(n+1)\le r,\qquad M_pr^p\ge e^{-L}p!\quad(p\ge0).
\]
The tail $\sum_{j>n}X_j$ is at most $r$. Requiring each of the
first $n$ variables to be at most $r/n$ therefore yields
\begin{align*}
\psi(2r)
&\ge\Prob(S\le2r)
 \ge\prod_{j=1}^n\min\{1,\mu_jr/n\}\\
&=\min_{0\le p\le n}M_p(r/n)^p
 \ge e^{-L}\min_{0\le p\le n}\frac{p!}{n^p}
 \ge e^{-L-n}\ge e^{-2L}.
\end{align*}
The product identity uses monotonicity of $\mu_j$, and the
penultimate inequality follows from $n^p/p!\le e^n$. Thus
\begin{equation}\label{eq:sharp-scale}
\psi\bigl(2e^{-R_M(L)}\bigr)\ge e^{-2L}\qquad(L\ge1).
\end{equation}
Since $R_M(L)\to\infty$ and $\psi$ is nondecreasing,
$\psi(t)>0$ for every $t>0$. Convexity gives
$\psi'(t)\ge\psi(t)/t>0$, so $\psi$ is strictly increasing on $[0,1]$.

For large $\lambda$, let $r_\lambda=\psi^{-1}(1/\lambda)\in(0,1)$.
Convexity and $\psi(0)=0$ imply
\[
\psi(t)\le\frac{t}{r_\lambda}\psi(r_\lambda)
\quad(0\le t\le r_\lambda),\qquad
\psi(t)\le t\psi'(t)\quad(t>0).
\]
Therefore
\[
\left|\int_0^{r_\lambda}
\frac{e^{i\lambda\psi(t)}-1}{t}\,dt\right|\le1.
\]
Since $1/(t\psi'(t))$ is decreasing, integration by parts gives
\[
\left|\int_{r_\lambda}^1\frac{e^{i\lambda\psi(t)}}t\,dt\right|
\le\frac2{\lambda r_\lambda\psi'(r_\lambda)}\le2.
\]
The phase vanishes on the negative half-line, so \eqref{eq:pv}
now gives
\begin{equation}\label{eq:sharp-identity}
\left|m_\psi(\lambda)+\log\frac1{r_\lambda}\right|\le3.
\end{equation}
Apply \eqref{eq:sharp-scale} with $L=(\log\lambda)/2$ to obtain
$r_\lambda\le2e^{-R_M((\log\lambda)/2)}$. Hence
\[
|m_\psi(\lambda)|
\ge R_M\bigl((\log\lambda)/2\bigr)-\log2-3
\ge\tfrac12R_M(\log\lambda)-5-\log2,
\]
where the last inequality is \eqref{eq:Rscale}.
Together with the upper bound, this proves \eqref{eq:dc-sharp}
for all sufficiently large frequencies.

\subsection{Proofs of the corollaries}

We use the notation $A(L)$ and $B(L)$ introduced above.
Multiplication of $M_n$ by a fixed geometric factor only adds a
constant to both entries, so we can use the stated sequences directly.
Changes in finitely many initial terms likewise do not affect the
following asymptotic orders.

For $M_n=(n!)^s$ with $s>1$, Stirling's formula and minimization at
$n\approx L$ give
\[
A(L)=(s-1)\log L+O(1).
\]
Since $\mu_n=n^s$ and
$T_M(N)\sim N^{1-s}/(s-1)$, the second entry satisfies
$B(L)=(s-1)\log L+O(1)$ as well. This proves
Corollary~\ref{cor:gevrey} for $s>1$. At $s=1$ the class is
quasianalytic, giving the asserted $O(1)$ bound.

For $M_n=n!Q_{k,s}(n)^n$ with $s>1$,
logarithmic differentiation gives
\[
\mu_n\sim nQ_{k,s}(n),\qquad
T_M(N)\sim\frac1{s-1}(\log^{(k)}N)^{1-s}.
\]
The second relation follows by an integral comparison and the
substitution $u=\log^{(k)}x$. Consequently
\[
B(L)=(s-1)\log^{(k+1)}L+O(1).
\]
The first entry is
$\inf_n\{L/n+\log Q_{k,s}(n)\}$, up to irrelevant initial terms.
Taking $n=\lceil L\rceil$ gives its upper bound; splitting the
lower bound into $n\le\sqrt L$ and $n\ge\sqrt L$ gives
\[
A(L)\approx\log Q_{k,s}(L)\approx\log\log L.
\]
Thus the minimum has order $\log^{(k+1)}L$, also when $k=1$.
At $s=1$, the series $\sum_n1/(nQ_{k,1}(n))$ diverges, so the
class is quasianalytic. This proves Corollary~\ref{cor:refined}.

For $M_n=\exp(cn^\alpha)$, where $c>0$ and $\alpha>1$, minimize
$L/n+cn^{\alpha-1}$ and use
$\log(n!)/n=O(\log n)$. Integer rounding does not change the
leading term, and
\[
A(L)\sim\alpha c^{1/\alpha}
(\alpha-1)^{-(\alpha-1)/\alpha}L^{(\alpha-1)/\alpha}.
\]
Here $\log\mu_n=c(n^\alpha-(n-1)^\alpha)
\sim c\alpha n^{\alpha-1}$. The elementary integral comparison
$\log\sum_{n\ge N}e^{-a n^\beta}\sim-aN^\beta$ for
$a,\beta>0$ therefore yields
\[
B(L)\sim c\alpha L^{\alpha-1}.
\]
The first entry is smaller in order. Substituting $L=\log\lambda$
proves Corollary~\ref{cor:fast}, for each fixed $c>0$.

Finally, let $M_n=\exp^{(k)}(cn^\alpha)$, with $k\ge2$ and
$c,\alpha>0$. Set
$a_n=\log M_n=\exp^{(k-1)}(cn^\alpha)$ and
$v=(\log^{(k-1)}L)^{1/\alpha}$. Choose integers
\[
n_- =\left\lfloor(2c)^{-1/\alpha}v\right\rfloor,
\qquad n_+ =\left\lceil(2/c)^{1/\alpha}v\right\rceil.
\]
Then $a_{n_-}=o(L)$ and $a_{n_+}/L\to\infty$.
Evaluation at $n_-$ gives $A(L)\lesssim L/v$.
For $n\le n_+$ the expression defining $A(L)$ is at least
$L/n_+-O(\log n_+)$, which is bounded below by a positive
constant times $L/v$. For $n\ge n_+$, since $a_n/n$ is eventually
increasing and $a_n/(n\log n)\to\infty$,
\[
\frac{a_n-\log(n!)}n\ge\frac{a_n}{2n}
\ge\frac{a_{n_+}}{2n_+}\ge\frac{L}{n_+}
\]
for large $L$. Thus
\[
A(L)\approx\frac{L}{(\log^{(k-1)}L)^{1/\alpha}}.
\]
Moreover, $a_{n+1}-2a_n+a_{n-1}\to\infty$, so
$\mu_{n+1}/\mu_n\to\infty$ and $T_M(N)\sim1/\mu_N$.
Since $(a_N-a_{N-1})/N\to\infty$, it follows that
$B(L)/L\to\infty$. The minimum is therefore again determined
by $A(L)$. Substitution of $L=\log\lambda$ proves
Corollary~\ref{cor:iterated}.

\end{document}